\documentclass[final,a4paper]{siamltex}

\usepackage[english]{babel} 
\usepackage{diagbox,amsmath,amssymb,amsfonts,mathrsfs,graphicx,subfigure}
\usepackage[ruled]{algorithm2e} 
\usepackage{myshortcuts}

\usepackage[pdftex,bookmarks=true, colorlinks=true, linkcolor=black, urlcolor=black, citecolor=black, linktoc=section,]{hyperref} 

\usetikzlibrary{decorations.pathreplacing}

\newcommand{\ej}[1]{#1}
\newcommand{\ejj}[1]{#1}
\newcommand{\er}[1]{#1}
\newcommand{\err}[1]{#1}
\newcommand{\jk}[1]{#1}
\newcommand{\jkk}[1]{#1}
\newcommand{\gm}[1]{#1}

\title{Sylvester-based preconditioning for the waveguide eigenvalue problem}
\author{Emil Ringh$^*$, Giampaolo Mele$^*$, Johan Karlsson$^*$, Elias Jarlebring\thanks{Dept. Mathematics, KTH Royal Institute of Technology, Lindstedtsvägen 25, Stockholm, Sweden, email: \{eringh, gmele, eliasj\}@kth.se, johan.karlsson@math.kth.se }}
\date{}

\begin{document}


\maketitle
\begin{abstract}
We \ej{consider a} nonlinear eigenvalue problem (NEP) arising from absorbing boundary
conditions in the study of a partial differential equation (PDE) describing a waveguide. 
We propose a new computational approach for this 
\gm{large scale}
NEP based on residual inverse iteration (Resinv) with preconditioned iterative solves.
\ej{Similar to many preconditioned iterative methods for discretized
  PDEs, this approach requires the construction of  }an accurate and efficient preconditioner. 
For the waveguide eigenvalue problem, 
the associated linear system \ej{ can be }formulated as a generalized Sylvester equation. The equation is approximated by a 
\gm{
low--rank correction} 
of a 
Sylvester equation, which we use as a preconditioner. 
The action of the preconditioner is efficiently computed 
using the matrix equation version of the Sherman-Morrison-Woodbury (SMW) formula.
\ej{We show how the preconditioner can be integrated into Resinv.}
  The results are illustrated by applying the method
to large-scale \ej{problems.}

\end{abstract}
\section{Introduction}\label{sec:intro}
We are concerned with the study of propagation of waves in a waveguide.
The application of two well established techniques
(Floquet theory and absorbing boundary conditions) leads to the following
characterization of wave propagation in $\RR^2$.
Details of such a derivation can be found\gm{, e.g.,} in \cite{Jarlebring:2015:WTIARTR, Tausch:2000:WAVEGUIDE}. 

The characterization is described by a PDE on a rectangular domain
$S_0=[x_-,x_+]\times [0,1]$. More precisely,
we \ej{wish to compute}
$u:S_0\rightarrow \CC$ and $\gamma\in\CC$ such that 
\refstepcounter{equation}
\label{eq:pde}\begin{align}
\Delta u(x,z) + 2\gamma u_z(x,z) + (\gamma^2+\kappa^2(x,z))u(x,z) &= 0 && (x,z)\in S_0\tag{\theequation a} \label{eq:pde_a} \\
u(x,0) &= u(x,1) && x\in(x_{-},x_{+})\tag{\theequation b} \label{eq:pde_b} \\
u_z(x,0) &= u_z(x,1) && x\in(x_{-},x_{+})\tag{\theequation c} \label{eq:pde_c} \\
\Top_{-,\gamma}[u(x_{-},\cdot)](z) &= -u_x(x_{-},z) && z\in(0,1) \tag{\theequation d} \label{eq:pde_d} \\
\Top_{+,\gamma}[u(x_{+},\cdot)](z) &= u_x(x_{+},z) && z\in(0,1) \tag{\theequation e} \label{eq:pde_e}.
\end{align}
The operators $\Top_{-,\gamma}$ and $\Top_{+,\gamma}$ are the so-called Dirichlet-to-Neumann (DtN) \er{maps},
which we specify in Section~\ref{sec:preliminaries}. The spatially dependent
constant $\kappa(x,z)$ is the wavenumber, \jk{which in our work is} 
assumed to be piecewise constant. A benchmark example is illustrated in Figure~\ref{fig:wg_challenge}.

\begin{figure}[t]
  \begin{center}
      \subfigure[Geometry of a benchmark waveguide.]{\includegraphics{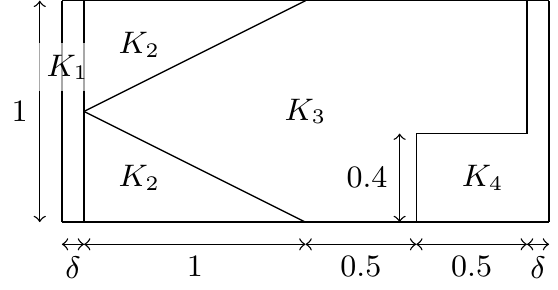}}
    \subfigure[The absolute value of an eigenfunction.]
    {\includegraphics{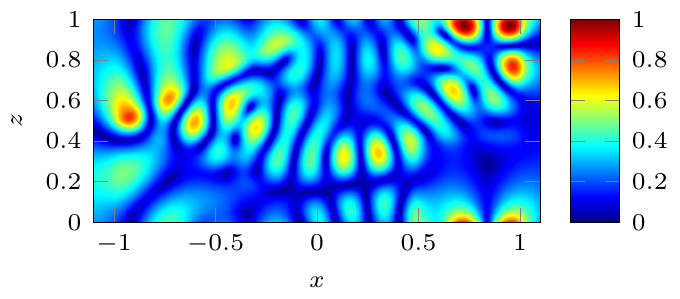}}
    \caption{
      Geometry of the benchmark waveguide and an eigenfunction corresponding to the eigenvalue $\gamma\approx -1.341 - 1.861i$. The same waveguide is used in the numerical examples, Section~\ref{sec:numerics}. The values $K_i$ indicates regions where the wavenumber $\kappa(x,z)$ is constant. For this waveguide $K_1 = \sqrt{2.3}\pi$, $K_2 = 2\sqrt{3}\pi$, $K_3 = 4\sqrt{3}\pi$, $K_4 = \pi$, and  $\delta=0.1$.
      \label{fig:wg_challenge}
    }
  \end{center}
\end{figure}

Note that \eqref{eq:pde} is a PDE-eigenvalue problem, where the eigenvalue 
$\gamma$
appears in a nonlinear way in the operator as well as \gm{in} the boundary conditions, due
to the $\gamma$-dependence of the DtNs. \ej{The problem \eqref{eq:pde} 
will be referred to as 
the waveguide eigenvalue problem (WEP)} and \gm{we} discretize
\ej{this} PDE in a way that allows us to construct an efficient iterative
procedure. More precisely, we derive results and methods with
a uniform finite-difference (FD) discretization, and also investigate its use in combination
with a finite-element method 
(FEM) discretization.
The \ej{discretization} 
is presented in Section~\ref{sec:discretization}.
Due to the nonlinearity in the PDE-eigenvalue problem, the
discretized problem is a nonlinear eigenvalue problem (NEP)
\ej{of} the form: find $(\gamma,v)\in\CC\times\CC^{n_zn_x+2n_z}\backslash\{0\}$ such that 
\begin{equation}  \label{eq:WEP}
M(\gamma)v = 0,
\end{equation}
where $n_x$ and $n_z$ \gm{are} the \ej{number of discretization} points in $x$- and $z$-direction
respectively. 

Over the last decades, the NEP has been considerably studied in the
numerical linear algebra community, and there is \er{a}
large family of different numerical methods, which \jk{we briefly summarize} 
as follows. A number of methods can be seen
as flavors of Newton's method, e.g.,
block Newton methods \cite{Kressner:2009:BLOCKNEWTON},
generalizations of inverse iteration \cite{Neumaier:1985:RESINV,Ruhe:1973:NLEVP}
and generalizations of the Jacobi-Davidson method. 
\ej{See PhD thesis \cite{Schreiber:2008:PHD} for a summary of these methods.}
A number of approaches are based on numerically computing
a contour integral \cite{Asakura:2009:NUMERICAL,Beyn:2011:INTEGRAL},
which can be accelerated as described in \cite{Xiao2016} and references therein.
Krylov methods and rational Krylov methods have
been generalized in various ways, e.g., the Arnoldi based methods
\cite{Voss:2004:ARNOLDI,Jarlebring:2012:INFARNOLDI},
rational Krylov approaches
\cite{Guttel:2014:NLEIGS,VanBeeumen:2013:RATIONAL,VanBeeumen:2015:CORK}.
See the summary papers \cite{Mehrmann:2004:NLEVP,Ruhe:1973:NLEVP,Voss:2013:NEPCHAPTER}
and \jk{the} benchmark collection \cite{Betcke:2013:NLEVPCOLL}
for further literature on methods for nonlinear eigenvalue problems.

\ejj{Most} \gm{of these methods involve}
the solution to
the associated linear system of equations
\begin{equation}  \label{eq:linsys}
  M(\sigma)y=r.
\end{equation}
For large-scale problems, the solution to this linear system
is often restricting the applicability of the method. 
In this work we adapt the method called residual inverse iteration
(Resinv) which was developed in \cite{Neumaier:1985:RESINV}.
Resinv is an iterative method
\gm{
for computing the  \ej{eigenvalue} closest to a given shift $\sigma \in \CC$
}
\ej{and it}
has the attractive feature that
\gm{the shift}
$\sigma$ 
is kept constant throughout
the iterations. 

The constant \er{shift allows} for precomputation, which 
\ej{reduces the computational effort for solving}
the linear systems  \eqref{eq:linsys}.
The standard \ej{way to exploit this} is to pre-compute 
\gm{ the}
LU-factorization of $M(\sigma)$.
\gm{
  Unfortunately this is not effective for \ej{our} large--scale
  problem,} \ejj{due to memory requirements.}
\ej{Instead,}
we propose to solve \eqref{eq:linsys}
with a preconditioned iterative method such as
GMRES \cite{Saad:1986:GMRES} or BiCGStab \cite{Vorst:1992:BICGSTAB},
where the constant shift and the structure of $M(\sigma)$
allows us to carry \jk{out} substantial precomputations \jk{in the}
initialization of Resinv. 

A number of recent approaches exploit 
that \ej{a} uniform discretization of \ej{a} rectangular domain
\er{PDE} can be \ej{expressed as a matrix equation, e.g., using
the  Sylvester equations or Lyapunov equations}.
The matrix-equation approach has been used, e.g., in the
setting of convection-diffusion equations \cite{Palitta:2015:MATRIX},
\er{fractional differential} equations \cite{Stoll2015},
PDE-constrained optimization \cite{Stoll2016},
and 
stochastic differential equations \cite{Powell:2015:BASIS}.
\jk{
Inspired by this, we propose a new preconditioner for the WEP 
based on \err{matrix equations}.
As a first step, shown in Proposition \ref{prop:Schur} and Section~\ref{sec:mateq},
the
\ej{linear system of equations} \eqref{eq:linsys} is formulated as a matrix equation. In Section~\ref{sec:SMW}, 
this matrix equation is approximated by} 
\gm{
a low-rank correction of a Sylvester equation, }
\jk{i.e.,  \ej{of} the form }
\begin{equation}  \label{eq:gensylv}
  \LLL(X)+\Pi(X)=C
\end{equation}
where $\LLL$ is a Sylvester operator and $\Pi$ is a \err{low-rank} linear operator
of the form $ \Pi(X):=\sum_{k=1}^N \Wop_k(X)E_k$, 
and \jk{where} $\Wop_k:\CC^{n\times n}\rightarrow\CC$, $k=1,\ldots,N$  are linear functionals.
We use a matrix equation version of the Sherman-Morrison-Woodbury (SMW) formula
to solve \eqref{eq:gensylv}. 
\jk{In Section~\ref{sec:precond} we describe how this can be done in a fast and memory efficient manner using the structures in our problem.}
A dominating part of the computation is independent
of $C$ and \jk{can therefore} 
be precomputed.
Properties of the approach are illustrated in Section~\ref{sec:numerics},
where we also compare the performance with other approaches.

The following notation is adopted in this paper.
We let $A \circ B$ denote the Hadamard,
or element-wise, matrix product between $A$ and $B$, and $A \kron B$ denotes the Kronecker product.
We let $\vec(A)\in\CC^{nm}$ denote the vectorization
of $A\in\CC^{n\times m}$, i.e.,  the vector obtained by
stacking the columns of $A$ on top of each other.
The set of eigenvalues of the matrix $A$ is denoted $\eig(A)$.
The $n\times n$ identity matrix is denoted  $I_n$. The
the matrix $J_n\in\RR^{n\times n}$ denotes the flipped identity matrix\gm{, that is $[J_{n}]_{k,\ell}=1$ if $k=n-\ell+1$ and 0 otherwise}.
The \jk{column vector} 
consisting of ones is denoted \jk{by} $\mathbf{1}$.

\section{Background and preliminaries}\label{sec:preliminaries}
\subsection{Problem background}
\ej{The PDE \eqref{eq:pde} stems from the propagation of
  waves in a periodic medium. The derivation can be briefly summarized as follows.
See \cite{Jarlebring:2015:WTIARTR, Tausch:2000:WAVEGUIDE} for details.}

Consider \ej{Helmholtz's} equation 
\begin{align}\label{eq:Helmhotz}
\Delta v(x,z) + \kappa(x,z)^2v(x,z) &= 0 && (x,z)\in\RR^2,
\end{align}
where 
\jk{
$\kappa(x,z)\in L^\infty(\RR^2)$ is the wavenumber. The wavenumber is a $1$-periodic function in the $z$-direction which is constant for sufficiently large $|x|$, i.e., $\kappa(x,z+1) = \kappa(x,z)$ for all $x,z$, and there exists real numbers $\xi_{-}$ and $\xi_{+}$ such that $\kappa(x,z) = \kappa_{-}$ for $x<\xi_{-}$ and $\kappa(x,z) = \kappa_{+}$ for $x>\xi_{+}$. }
\gm{We use}
\ej{the ansatz}
\begin{align*}
v(x,z) = e^{\gamma z}u(x,z)
\end{align*}
where $u(x,z+1) = u(x,z)$ in \eqref{eq:Helmhotz} and 
\gm{ we apply}
absorbing boundary conditions at $x=x_-\leq\xi_-$ and $x=x_+\geq\xi_+$. 
\ej{From this ansatz we} directly identify that $u$ satisfies \eqref{eq:pde_a}.
A more precise analysis (presented in \cite{Jarlebring:2015:WTIARTR}) shows
that also \eqref{eq:pde_d}-\eqref{eq:pde_e}
are satisfied where the DtN-maps are defined by
\begin{align}\label{eq:dtn}
\Top_{\pm,\gamma}[g](z) := \sum_{k\in\ZZ} s_{\pm,k}(\gamma)g_k e^{2\pi i k z},
\end{align}
where $\{g_k\}_{k\in\ZZ}$ are
the Fourier series coefficients of the function $g(z)$ and $s_k$, $k\in\ZZ$\er{,} are given by 
\refstepcounter{equation}
\label{eq:dtn_coef}
\begin{align}
s_{\pm,k}(\gamma) & := \sign(\im(\beta_{\pm,k}(\gamma))) i \sqrt{\beta_{\pm,k}(\gamma)} \tag{\theequation a} \label{eq:dtn_coef_a}\\
\beta_{\pm,k}(\gamma) & := (\gamma + 2\pi i k)^2 + \kappa^2_\pm \tag{\theequation b} \label{eq:dtn_coef_b}.
\end{align}

\subsection{Discretization of the WEP}\label{sec:discretization}
The PDE \eqref{eq:pde} is in this work discretized as follows.
We use a uniform \ej{FD} discretization with $n_x$ and $n_z$ \gm{points} in $x$- and $z$-direction respectively. The grid consists of the points 
\jk{$x_k = x_- + kh_x$ for $k=1,2,\dots n_x$ where $h_x=(x_+ - x_-)/(n_x+1)$,
and $z_\ell = \ell h_z$ for  $\ell=1,2,\dots,n_z$ where $h_z=1/n_z$.}

\ej{This FD-discretization leads to the NEP} \er{\eqref{eq:WEP}} \jk{being} described by the following block matrix
\begin{align}\label{eq:M_FD}
M(\gamma) := \begin{pmatrix}
Q(\gamma) & C_1 \\
C_2^\trans & P(\gamma)
\end{pmatrix}.
\end{align}
\jk{Here $Q(\gamma)$ represents the discretization of the interior \eqref{eq:pde_a} and $P(\gamma)$ represents the Dirichlet-to-Neumann maps, \eqref{eq:pde_d} and \eqref{eq:pde_e}. The matrix $C_1$ represents the effect of the
boundary points to the interior and $C_2^\trans$ represents the effect of the
interior on the boundary constraints, i.e., \eqref{eq:pde_d} and \eqref{eq:pde_e}.}
\jk{The matrix}
$Q(\gamma)\in\CC^{n_xn_z\times n_xn_z}$ is 
\ej{large}
and sparse, \ej{and}  given by
\begin{align}\label{eq:Q_FD}
Q(\gamma) := A_0 + \gamma A_1 + \gamma^2 A_2,
\end{align}
with $A_0 := D_{xx}^\trans \kron I_{n_z} + I_{n_x}\kron D_{zz} + \diag(\vect(K))$, and $A_1 := 2 I_{n_x}\kron D_{z}$, and $A_2 :=I_{n_x n_z}$.
\jkk{Here $D_{xx}\in \RR^{n_x\times n_x}$ is the second derivative matrix, and $D_{z}, D_{zz}\in \RR^{n_z\times n_z}$ are the circulant first and second derivative matrices. 
That is, $D_{xx}=(-2I_{n_x}+Z_{n_x}+Z_{n_x}^T)/h_x^2$, $D_{z}=(Z_{n_z}+e_1e_{n_z}^T-Z_{n_z}^T-e_{n_z}e_{1}^T)/(2h_z)$, and $D_{zz}=(-2I_{n_z}+Z_{n_z}+e_1e_{n_z}^T+Z_{n_z}^T+e_{n_z}e_{1}^T)/h_z^2$,
where $Z_n\in \RR^{n\times n}$ is the shift matrix, defined by $[Z_n]_{k,l}=1$ if $k-l=1$ and $0$ otherwise.
}

The matrix $K$ \jkk{is} the discretization of the \er{squared wavenumber,
i.e.,  $[K]_{k,\ell} := \kappa^2(x_\ell,z_k)$.} 
The block $C_1\in\CC^{n_xn_z\times 2n_z}$ is given by
\begin{align}\label{eq:C1_FD}
C_1 := \frac{1}{h_x^2}\left(
e_1\kron I_{n_z},\
e_{n_x}\kron I_{n_z}
\right) ,
\end{align}
and the block $C_2^\trans \in\CC^{2n_z\times n_xn_z}$ \ej{is given by}
\begin{align}\label{eq:C2_FD}
C_2^\trans := \begin{pmatrix}
d_1 e_{1}^\trans\kron I_{n_z} +d_2 e_{2}^\trans\kron I_{n_z} \\
d_1 e_{n_x}^\trans\kron I_{n_z} +d_2 e_{n_x-1}^\trans\kron I_{n_z}
\end{pmatrix} ,
\end{align}
with $d_1 := \frac{2}{h_x}$ $d_2 := -\frac{1}{2h_x}$.
The last block $P(\gamma)\in\CC^{2n_z\times2n_z}$ stems from the discretization of the DtN operators,
i.e.,  the operators in the left hand side of \eqref{eq:pde_d} and \eqref{eq:pde_e}. We truncate the Fourier series expansion in \eqref{eq:dtn} and use only the \ej{coefficients corresponding to  
$-p,\ldots,p$}, 
we choose $p$ such that $n_z = 2p +1$. Then
\begin{align}\label{eq:P}
P(\gamma)
:= \begin{pmatrix}
P_{-}(\gamma) & 0\\
0 & P_{+}(\gamma)
\end{pmatrix}
= \begin{pmatrix}
R\Lambda_-\ej{(\gamma)}R^{-1} & 0\\
0 & R\Lambda_+\ej{(\gamma)}R^{-1}
\end{pmatrix}
\end{align}
with $\Lambda_\pm := \diag(S_\pm)$, where $S_\pm := \left( s_{\pm,-p} + d_0,\ s_{\pm,-p+1} + d_0,\ \dots,\ s_{\pm,p} + d_0 \right)$, and $d_0 := -\frac{3}{2h_x}$, and where $s_{\pm,k}$ are defined by \eqref{eq:dtn_coef_a}. Moreover,
$[R]_{k,\ell}= e^{2\pi i(\ell-p-1)kh_z}$
%
and $R^\herm = n_z R^{-1}$.
\begin{remark}\label{rem:action_of_P}
  \ej{Note} that $R$ is the Fourier matrix left multiplied with the anti-diagonal matrix with $(1,e^{2\pi iph_z},e^{2\pi i2ph_z},\dots,e^{2\pi (n_z-1)ph_z})$ on its anti-diagonal. \er{Consequently, the} action of both $R$ and $R^{-1}$ on a vector can be efficiently calculated with the Fast Fourier Transform \er{(FFT)
  and} from \eqref{eq:P} we conclude 
  that calculating the action of $P_-(\gamma)$, $P_-(\gamma)^{-1}$, $P_+(\gamma)$, and $P_+(\gamma)^{-1}$ on a vector can be done in $\OOO(n_z\log(n_z))$ operations. 
\end{remark}

\ej{For future reference, we now also note} that when $\gamma$ is in the left half-plane of $\CC$ the derivative of $M(\gamma)$ with respect to $\gamma$ is given by 
\begin{align}\label{eq:M_prime_FD}
M'(\gamma) := \begin{pmatrix}
Q'(\gamma) & 0\\ 0 & P'(\gamma)
\end{pmatrix},
\end{align}
where $Q'(\gamma) := A_1 + 2\gamma A_2$ and
\jk{
$P'(\gamma)=\diag(R\Lambda_-'\ej{(\gamma)}R^{-1}, R\Lambda_+'\ej{(\gamma)}R^{-1}).$
The matrices are directly given by}
$\Lambda_\pm' := \diag(S_\pm')$, where $S_\pm' := \left( s_{\pm,-p}' ,\ s_{\pm,-p+1}',\ \dots,\ s_{\pm,p}' \right)$,
\jk{$s'_{\pm,k}(\gamma)  := \sign(\im(\beta_{\pm,k}(\gamma))) i (\gamma+2\pi i k)/ \sqrt{\beta_{\pm,k}(\gamma)}$,}
and $\beta_{\pm,k}(\gamma)$ are given by \eqref{eq:dtn_coef_b}.

\subsection{Residual inverse iteration for the WEP}\label{sec:resinv}
Our approach is based on the Resinv \cite{Neumaier:1985:RESINV}
as a solution method for the NEP \eqref{eq:WEP} with $M$ defined
by \eqref{eq:M_FD}.
Given an approximation to the eigenpair \ej{$(\gamma_k,v_k)$,}
  Resinv iteratively computes new approximations in each iteration. The procedure consists of a few steps. The
first step \jk{is to compute} 
a new approximation of the eigenvalue $\gamma_{k+1}$ by solving the nonlinear scalar equation
\begin{align}  \label{eq:Rayleigh_quot}
 v_{k}^\herm M(\gamma_{k+1})v_k = 0.
\end{align}
There are different ways of choosing the left vector in \eqref{eq:Rayleigh_quot} discussed in the literature \cite{Neumaier:1985:RESINV,Jarlebring:2011:RESINVCONV,Schreiber:2008:PHD}, but we choose the current approximation of the right eigenvector, as it is presented in the equation.
\ej{Equation \eqref{eq:Rayleigh_quot} is solved with} \er{Newton's method}
\ej{in one unknown variable}
which requires that we calculate the derivative of $\ej{v_k^*}M(\gamma)\ej{v_k}$
with respect to $\gamma$. The derivative, for  $\gamma$ in the left half-plane of $\CC$,
\ej{can be computed from} \eqref{eq:M_prime_FD}.
The second step consists of computing a residual 
\begin{align}\label{eq:residual_resinv}
  r_{k}=M(\gamma_{k+1})v_k.
\end{align}
Subsequently $r_k$ is used to calculate a correction to the eigenvector by solving
\begin{align}  \label{eq:Delta_vk}
  \Delta v_{k}=M(\sigma)^{-1}r_k,
\end{align}
where  $\sigma$ is a fixed shift that is used throughout the whole procedure. We propose to use a preconditioned iterative method to \er{solve \eqref{eq:Delta_vk}.} The third and last step is to update the eigenvector approximation $v_{k+1}=v_{k}-\Delta v_k$, and to normalize it $v_{k+1} = v_{k+1}/\norm{v_{k+1}}$. The Resinv procedure is summarized in Algorithm \ref{alg:resinv}.
\begin{algorithm} 
\caption{Resinv with preconditioned iterative solves}\label{alg:resinv}
\SetKwInOut{Input}{input}\SetKwInOut{Output}{output}
\Input{Initial guess of the eigenpair $(\gamma_0,v_0)\in\CC\times\CC^{n_xn_z+2n_z}$, with $\norm{v_0} = 1$}
\Output{An approximation $(\gamma,v)\in\CC\times\CC^{n_xn_z+2n_z}$ of  $(\gamma_*,v_*)\in\CC\times\CC^{n_xn_z+2n_z}$}
\BlankLine 
\nl \For{$k = 0,1,2,\dots$}{
\nl Compute new approximation of $\gamma_{k+1}$ from \eqref{eq:Rayleigh_quot}\\
\nl Compute the residual $r_k$ from \eqref{eq:residual_resinv}\\
\nl Compute the correction $\Delta v_k$ from \eqref{eq:Delta_vk} with a preconditioned iterative method.\label{alg_step:resinv_lin_solve}\\
\nl $v_{k+1} \gets v_k - \Delta v_k$\\
\nl $v_{k+1} \gets v_{k+1}/\|v_{k+1}\|$\\
}
\nl $\gamma \gets \gamma_k$, \ \ $v \gets v_k$\\
\end{algorithm}

Large parts of the computational effort in Algorithm \ref{alg:resinv}
\ej{often consists of the solving of} the linear system \eqref{eq:Delta_vk} and
\ej{we present a} method that makes the
computation feasible for \err{large-scale} problems. We use the Schur complement of $M(\sigma)$ with respect to the block $P(\sigma)$,
\begin{align}\label{eq:Schur}
S(\sigma):= Q(\sigma) - C_1 P(\sigma)^{-1}C_2^\trans,
\end{align}
to specialize the computation of \eqref{eq:Delta_vk} for the WEP. The specialization is an important step \jk{in our algorithm} and therefore we present it in the following \jk{form.}
\begin{proposition}[Schur complement for the WEP]\label{prop:Schur}
Let $M(\sigma)$ be as in \eqref{eq:M_FD}, the \er{shift $\sigma\in\CC$,} and let $S(\sigma)$ be the Schur complement \eqref{eq:Schur}. 
Moreover, let $r\in\CC^{n_xn_z + 2n_z}$, and let $r_\text{int}$ be the first $n_xn_z$ elements of $r$ and $r_\text{ext}$ be the last $2n_z$ elements of $r$. 
Then
\begin{align} \label{eq:M_inverse}
M(\sigma)^{-1}r=\begin{pmatrix}
q\\
P(\sigma)^{-1}\left(-C_2^\trans q + r_{\text{ext}}\right)
\end{pmatrix}
\end{align}
where
\begin{align}  \label{eq:Ssigma_solve}
q :=S(\sigma)^{-1} \tilde{r} 
\end{align}
and
\begin{align}  \label{eq:Ssigma_solve_rhs}
\tilde{r} := r_{\text{int}} - C_1 P(\sigma)^{-1}r_{\text{ext}}.
\end{align}
\end{proposition}

We \gm{use} Proposition~\ref{prop:Schur} 
\gm{to solve the linear system in}
Step~4 in Algorithm~1.
More precisely, we use a preconditioned iterivative method to solve \eqref{eq:Ssigma_solve},
and FFT to compute the action $P(\sigma)^{-1}$
(as described in Remark~\ref{rem:action_of_P}).
All other operations required for the application proposition have
negligable computational cost.

%

\section{Matrix equation characterization}\label{sec:mateq}
In order to construct a good preconditioner for the linear system \eqref{eq:Ssigma_solve}
we now \jk{formulate} 
it \jk{as a matrix equation.} 
Without loss of generality we express \eqref{eq:Ssigma_solve}
as $S(\sigma)\vec(X)=\vec(C)$\jk{, where $X, C\in\CC^{n_z\times n_x}$}.

Note that $S(\sigma)$ is defined in \eqref{eq:Schur} as
the sum of $Q(\sigma)$ and $-C_1P(\sigma)^{-1}C_2^T$,
where 
$Q(\sigma)$ \er{is} described by \eqref{eq:Q_FD}.
The action of $Q(\sigma)$ can be characterized
with matrix equations.
By direct application of rules for Kronecker products, see e.g., \cite[Section~4.3]{Horn:1991:MATAN2_edit}, \jk{it follows} 
that
\begin{align}\label{eq:QMATEQ}
Q(\sigma)\vec(X)=\vec\left((D_{zz}+2\sigma D_z+\sigma^2 I_{n_z})X + XD_{xx} + K\circ X\right).
\end{align}
\jk{The action of the first two terms of \eqref{eq:QMATEQ} can be identified with a Sylvester operator,
 $\Lop: \CC^{n_z\times n_x}\rightarrow\CC^{n_z\times n_x}$:
\begin{equation}  \label{eq:Ldef}
\Lop(X) := AX + XB,
\end{equation}
and hence the action of $Q(\sigma)$ can be viewed as a generalized Sylvester operator.}
The action corresponding to $S(\sigma)$ can
similarly also be constructed as a generalization 
of the Sylvester operator. We formalize it in the following
result, where we \jk{also introduce} 
an additional free parameter $\bar{k}$.
This parameter is later chosen \jk{in} such a way that the contribution of
the terms corresponding to the Sylvester operator
\gm{
is large.}

\begin{proposition}[Waveguide matrix equation]\label{prop:WMATEQ}
Let $X\in\CC^{n_z\times n_x}$, let $C\in\CC^{n_z\times n_x}$ be a given matrix, and let $S(\sigma)$ be \er{the Schur complement} \eqref{eq:Schur}. Then $\vec(X)$ is a solution to $S(\sigma)\vec(X) = \vec(C)$ if and only if $X$ is a solution to
\gm{
\begin{align}  \label{eq:WMATEQ}
\begin{aligned}
&A X+XB+
(K-\bar{k}\mathbf{1}\mathbf{1}^\trans)\circ X
-P_-(\sigma)^{-1}XE-P_+(\sigma)^{-1}XJ_{n_x}EJ_{n_x} 
= C,
\end{aligned}
\end{align}
}
where 
$A:=D_{zz}+2\sigma D_z+\sigma^2 I_{n_z}+\bar{k}I_{n_z}$, and $B:=D_{xx}$, and $E:=\frac{1}{h_x^2}(d_1 e_1 + d_2 e_2)e_1^\trans$, $d_1$and $d_2$ is given by the discretization, $J_{n_x}$ is the flipped identity, and $\bar{k}$ is a free parameter.
\end{proposition}
\begin{proof}
We have that $S(\sigma)= Q(\sigma)-C_1 P(\sigma)^{-1}C_2^\trans$. The equivalent matrix equation formulation for $Q(\sigma)$ is found \er{apparent from} \eqref{eq:QMATEQ}. The rest follows from the calculation
\begin{align*}
\begin{aligned}
& C_1 P(\sigma)^{-1}C_2^\trans \\
&=
\frac{1}{h_x^2}\begin{pmatrix}
e_1\otimes I_{n_z},\
e_{n_x}\otimes I_{n_z}
\end{pmatrix}
\begin{pmatrix}
P_{-}(\sigma)^{-1} & 0\\
0 & P_{+}(\sigma)^{-1}
\end{pmatrix}
\begin{pmatrix}
d_1 e_{1}^\trans\otimes I_{n_z} +d_2e_{2}^\trans\kron I_{n_z} \\
d_1 e_{n_x}^\trans\otimes I_{n_z} +d_2e_{n_x-1}^\trans\kron I_{n_z}
\end{pmatrix} \\
&=
e_1\left(\frac{d_1}{h_x^2}  e_{1}^\trans +\frac{d_2}{h_x^2} e_{2}^\trans\right)\kron P_{-}(\sigma)^{-1} +
e_{n_x}\left(\frac{d_1}{h_x^2}  e_{n_x}^\trans +\frac{d_2}{h_x^2} e_{n_x-1}^\trans\right)\kron P_{+}(\sigma)^{-1} .
\end{aligned}
\end{align*}
\end{proof}


\section{The Sylvester SMW structure for the WEP}\label{sec:SMW}
\subsection{Sylvester-type SMW-structure}\label{subsec:SMW}
Our computational procedure is based on the explicit
formula for the inverse of a matrix with a  \err{low-rank}
\gm{correction},
the Sherman-Morrison-Woodbury (SMW) formula
\cite[Equation (2.1.4)]{Golub:2013:MATRIX_edit}.
We use the formulation
\begin{equation}  \label{eq:SMW}
   (L+UV^T)^{-1}c=L^{-1}(c-UW^{-1}V^T L^{-1}c)  
\end{equation}
where $U,V\in\CC^{n\times N}$ and
\begin{equation}  \label{eq:Wdef}
  W:=I+V^TL^{-1}U \in\CC^{N\times N}.
\end{equation}
\ejj{In order to apply}
the SMW-formula to equations of the
\jk{form \eqref{eq:WMATEQ}}, 
we need a particular matrix equation version of the
SMW-formula.
The adaption of SMW-formula's to matrix equations 
has been examined previously in the literature
\cite{Damm:2008:ADIPRECONDITIONED_edit,Kuzmanovic:2013:Sherman–Morrison–Woodbury,Richter:1993:EFFICIENT}. 
Our formulation is based on a specialization
of \cite[Lemma~3.1]{Damm:2008:ADIPRECONDITIONED_edit}
\jk{that is set up to} 
minimize the memory requirements
(as we further discuss in Remark~\ref{rem:SMW-variant}).

We select the $L$-matrix in \eqref{eq:SMW} as
the vectorization of a Sylvester operator
\eqref{eq:Ldef}, 
which is invertible if $\eig(A)\cap\eig(-B) = \emptyset$, see e.g., 
\cite[Theorem 4.4.6]{Horn:1991:MATAN2_edit}.
We make this specific choice \ejj{since}
\jk{ the \ejj{solution to the} Sylvester equation \ejj{in our case be computed} efficiently.}
\ejj{More precisely, the} specific structure \ejj{present in our context}
can be exploited, as we further describe in
Section~\ref{subsubsec:fft}.

In our approach we consider \gm{a} rank $N$ 
\gm{correction}
of the Sylvester operator, which 
can be expressed as a 
linear operator $\Pi$ of the form
\begin{equation}  \label{eq:Pidef2}
    \Pi(X):=\sum_{k=1}^N \Wop_k(X)E_k.
  \end{equation} 
In this setting the matrix $W$ in \eqref{eq:Wdef} can be expressed in terms of evaluations of the 
functionals $\Wop_1,\ldots,\Wop_N$.  
This \jk{use of SMW} is formalized in the following result.

\begin{theorem}[Sylvester-type SMW-structure]\label{thm:SMW}
  Let $A\in\CC^{n\times n}$, $B\in\CC^{m\times m}$, and $C\in\CC^{n\times m}$
  and suppose $\eig(A)\cap\eig(-B) = \emptyset$.
  Moreover, 
  let the matrices $E_k\in\CC^{n\times m}$ and linear functionals $\Wop_k : \CC^{n\times m}\rightarrow \CC$ be given for $k=1,2,\dots,N$
  and define $\Pi:\CC^{n\times m}\rightarrow\CC^{n\times m}$ by \eqref{eq:Pidef2}.
Assume that there exists a unique solution to the equation
\begin{align}\label{eq:SMW-structure}
  \Lop(X)+\Pi(X)= 
  C .
\end{align}
where $\Lop$ is the Sylvester operator defined analogous to  \eqref{eq:Ldef}.
Moreover, let 
\begin{subequations}
\begin{eqnarray}
G &:=& \Lop^{-1}(C), \qquad\textrm{ and }\label{eq:SMW_G}\\
F_k &:=& \Lop^{-1}(E_k) \qquad \text{for } k=1,2,\dots,N.\label{eq:SMW_F}
\end{eqnarray}
\end{subequations}
and define 
\begin{equation}  \label{eq:Wdef2}
  W:=
  \begin{pmatrix}
1 + \Wop_1(F_1)& \Wop_1( F_2 ) & \dots & \Wop_1( F_N ) \\
\Wop_2( F_1) & 1 + \Wop_2( F_2 ) & \dots & \Wop_2( F_N ) \\
\vdots & \vdots & \ddots & \vdots \\
\Wop_N( F_1) & \Wop_N( F_2 ) & \dots & 1 + \Wop_N( F_N ) \\
  \end{pmatrix},\;\;
 g:=\begin{pmatrix}
\Wop_1( G ) \\
\Wop_2(G ) \\
\vdots \\
\Wop_N( G )
\end{pmatrix}.
\end{equation}
Then the solution to \eqref{eq:SMW-structure} is given by
\begin{align}\label{eq:SMW-solution}
X = \Lop^{-1}\left( C - \sum_{k=1}^N \alpha_k E_k\right),
\end{align}
where $a^T=(\alpha_1,\ldots,\alpha_N)$ 
is  the unique solution to the system of equations
\begin{equation}  \label{eq:alpha_system}
  W a =g.
\end{equation}
\end{theorem}
\begin{proof}
  In order \er{to} invoke
  \cite[Lemma~3.1]{Damm:2008:ADIPRECONDITIONED_edit}
  we note that the linear functionals $\Wop_k$,
  for $W_k\in\CC^{n\times m}$, \gm{can}
  be parametrized as $\Wop_k(X) =\vec\left(W_k\right)^\trans \vec(X)$.
  Moreover, we define the matrices
  $P_1:=[\vec(E_1),\ldots,\vec(E_N)]$, and
  $P_2:=[\vec(W_1),\ldots,\vec(W_N)]^T$.
  The conclusion \jk{\eqref{eq:SMW-solution}-}\eqref{eq:alpha_system}
  follows from direct reformulation of
  \cite[Equation~(6)]{Damm:2008:ADIPRECONDITIONED_edit}.
\end{proof}
\begin{remark}[Variants of Theorem~\ref{thm:SMW}] \label{rem:SMW-variant}
  Note that,
  due to the linearity of $\Lop^{-1}$, 
the solution in \eqref{eq:SMW-solution} can be equivalently expressed as
\begin{align} \label{eq:SMW-solution2}
X = G - \sum_{k=1}^N \alpha_k F_k.
\end{align}
Moreover, since $G,F_1,\dots,F_N$ can be treated as  known,
\ej{$X$ can be computed directly from  
 \eqref{eq:SMW-solution2} without the action of $\Lop^{-1}$}.
 Hence, an approach based on \eqref{eq:SMW-solution2}
requires less computational effort than an
approach based on \eqref{eq:SMW-solution} in general.
However, in our case, \eqref{eq:SMW-solution2}
is not advantageous
\ej{since it requires more memory resources} as we further discuss in 
Section~\ref{subsec:storage}.
\end{remark}


\subsection{SMW-structure approximation of \er{the} waveguide matrix equation}\label{subsec:approx}
We saw in the previous section that
the matrix equation SMW-formula can be applied to
sums of a Sylvester operator and the operator $\Pi$
in \eqref{eq:Pidef2}. Note that any linear matrix-operator
can be expressed 
in the form \eqref{eq:Pidef2},
by selecting
the functionals as $\Wop_k(X)=e_j^TXe_\ell$ and
the matrices $E_k=e_j e_\ell^T$, where
$j=1,\ldots,n$, and $\ell =1,\ldots,m$, and $k=j+(\ell-1)n$
such that $k\in \{1,\ldots,nm\}$ and $N=nm$.
Unfortunately, such a construction is not practical
since Theorem~\ref{thm:SMW}
is not computationally attractive for large values of $N$.

\gm{
The equation \eqref{eq:WMATEQ} can be efficiently solved if the last terms in \eqref{eq:WMATEQ}, i.e,  }
\begin{align}  \label{eq:Rest_WMATEQ}
\begin{aligned}
\Phi(X):=(K-\bar{k}\mathbf{1}\mathbf{1}^\trans)\circ X
-P_-(\sigma)^{-1}XE-P_+(\sigma)^{-1}XJEJ,
\end{aligned}
\end{align}
can be expressed 
as a \err{low-rank} operator $\Pi$ of the form \eqref{eq:Pidef2}. Then 
the solution to \eqref{eq:WMATEQ}
can be directly computed with Theorem~\ref{thm:SMW}.
\gm{
In general, $\Phi$ 
}
can 
 only be expressed
as an operator $\Pi$ of large rank $N$. 
However, there exists some situations
where exact matching can be achieved with $N\ll n_xn_z$,
for instance if the elements of $K$ equals a constant, $\bar k$, except for a few indices. 
In
the continuous formulation, this
corresponds to the wavenumber being constant
in most parts of the domain.

Although, $\Phi$  can in general only be expressed
in the form of \eqref{eq:Pidef2} with a large $N$,
we now 
\gm{introduce}
a \err{low-rank} approximation of $\Phi$, i.e., with $N\ll n_xn_z$.
Our construction exploits the structure in $\Phi$ and allows for a representation of $\Pi$ with both low rank $N$ and 
\gm{structured}
matrices $E_k$. 

We consider approximations in a vector space $\mathcal{V}\subset\CC^{n_z\times n_x}$,
with a basis $V_1,\ldots,V_N$, which is assumed to be orthogonal
with respect to the trace inner product
$\langle X,Y\rangle=\trace(Y^*X)$. We take approximation of $X\in\CC^{n_z\times n_x}$ from
this space and let $\tilde{X}$ be the best approximation
(in the induced trace norm). 
Equivalently we can impose the Galerkin condition on $X$,
\[
   \langle X-\tilde{X},V_k \rangle=0,\quad \textrm{ for  }k=1,\ldots,N,
\]
which leads to the formula $\tilde{X}:=\sum_{k=1}^N\frac{\langle X,V_k\rangle}{\langle V_k,V_k\rangle}V_k$. 
Based on this approximation, we construct $\Pi$ as an
approximation of $\Phi$ by setting
\[
\Pi(X):=\Phi\left(\sum_{k=1}^N\frac{\langle X,V_k\rangle}{\langle V_k,V_k\rangle}V_k\right)=
\sum_{k=1}^N\frac{\langle X,V_k\rangle}{\langle V_k,V_k\rangle}\Phi(V_k).
\]
\gm{ 
If we define $\Wop_k(X):=\frac{\langle X,V_k\rangle}{\langle V_k,V_k\rangle}$ 
and $E_k:=\Phi(V_k)$, then $\Pi$ is of the form \eqref{eq:Pidef2}.
More precisely,} for our structure in \eqref{eq:Rest_WMATEQ} we have
\begin{align}\label{eq:Ek_lm}
E_{k} := (K-\bar{k}\mathbf{1}\mathbf{1}^\trans)\circ V_{k}
-P_-(\sigma)^{-1} V_kE-P_+(\sigma)^{-1} V_kJEJ.
\end{align}
As can be expected from a Galerkin approach,
the approximation is exact for any
$X\in\mathcal{V}$, since by construction $\Phi(V_k)=\Pi(V_k)$, $k=1,\ldots,N$.

In theory, the construction can be done for any appropriate
vector space. For reasons of structure exploitation, we
select $V_k$, $k=1,\ldots,N$, 
as indicator function\jk{s} in rectangular regions, 
as shown in Figure \ref{fig:SMW-grid-for-S}.
We \jk{then} select $N_x$ and $N_z$ intervals in $x$- and $z$-direction
respectively, \jk{hence} 
$N=N_xN_z$.
In this case the matrices $V_{\ell+m(N_z-1)}$,
$\ell=1,\ldots,N_z$ and $m=1,\ldots,N_x$, 
take the value 1 in the corresponding rectangular region and zero outside,
and  $\Wop_{\ell+m(N_z-1)}$ is the functional taking the mean over that region.
More precisely, 
\begin{align*}
[V_{\ell+m(N_z-1)}]_{p,q} &= \begin{cases} 
   1 & \text{if } (p,q) \text{ belongs to the region } (\ell,m)\\
   0 & \text{otherwise}
  \end{cases}\\
\Wop_{\ell+m(N_z-1)}(X) &= \frac{\sum_{p,q} [X\circ [V_{\ell+m(N_z-1)}]]_{p,q}}{\sum_{p,q} [V_{\ell+m(N_z-1)}]_{p,q}}.
\end{align*}
Note that 
\jk{the grid resolution is finer}
near the boundary. This
is done in order to improve the approximation 
\jk{of the DtN-map\err{s}, which \err{are} localized in the boundary region.}
\er{The} localization can be seen in
the structure of the $E$-matrix in 
Proposition~\ref{prop:WMATEQ}.
This choice is supported by 
computational experiments presented in Section~\ref{sec:numerics}.

\begin{remark}[Other approximations]
  Note that the approximation described above
  is only an illustration of an approximation
  procedure and there exist many variations. Apart
  from other \er{Galerkin type} approaches,
  it is also possible to use \er{the}
  rank-revealing procedure proposed in
  \cite[Algorithm~3.2]{Damm:2008:ADIPRECONDITIONED_edit}.
  Another option is to use smoothing as, e.g.,
  in multi-grid methods and
  other domain decomposition methods \cite{Dolean:2015:DOMAINDECOMP}. 
  \gm{In order to use these approaches in the framework here described, 
    \ejj{further focused research would be required.}
    In our setting, }
  the structure of  \eqref{eq:Ek_lm}
  allows us to reduce the memory requirements
  as we describe in the   next section. 
\end{remark}

\begin{figure}[h]
  \begin{center}
    \includegraphics{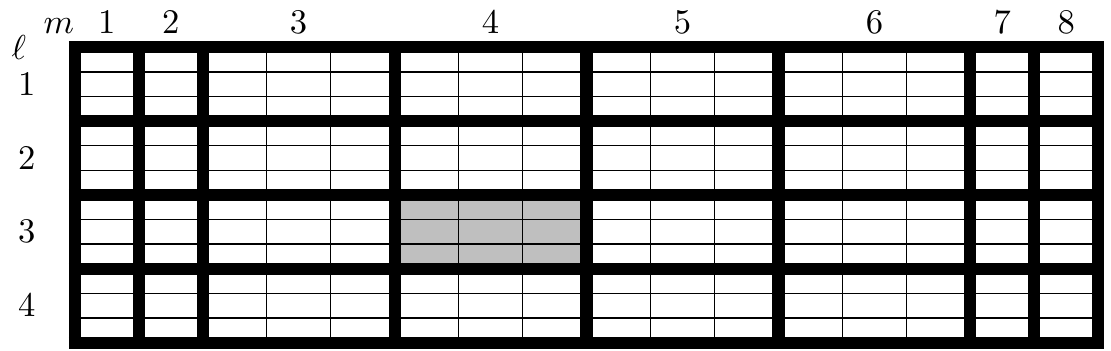}
    \caption{
      The coarse grid shows the regions in which we take basis vectors $V_k$ as constant. The fine grid is an illustration of the underlying grid from the discretization of the PDE.
      For example the weight matrix $V_{3+4(N_z-1)}$ will have value 1 in the light gray area and 0 on all other elements.
      Note that the coarse grid for the SMW approximation (SMW-grid) is finer towards the boundaries in $x$-direction, this is to better capture the effect of the DtN-maps.
      \label{fig:SMW-grid-for-S}
    }
  \end{center}
\end{figure}

\section{Structure exploitation and specialization of  Resinv}\label{sec:precond}

\subsection{SMW-preconditioned Resinv}\label{subsec:smw_resinv}
The application of Resinv to the WEP, described in Section~\ref{sec:resinv},
requires an efficient solution to the linear system $M(\sigma)^{-1}r$
in equation \eqref{eq:Delta_vk}. Since we want to
solve this linear system iteratively, we need an effective
preconditioner.
We use the approximation technique presented in Section~\ref{subsec:approx}
as a preconditioner. As a consequence of
the fact \jk{that} the shift is kept constant in Resinv, the
Sylvester operator defined in \eqref{eq:Ldef},
with \er{$A:=D_{zz}+2\sigma D_z+\sigma^2 I_{n_z}+\bar{k}I_{n_z}$ and $B:=D_{xx}$}, is also constant.
Therefore the matrices $F_1,\ldots,F_N$ in
\eqref{eq:SMW_F}
and the
$W$-matrix in \eqref{eq:Wdef2}
are constant over the iterations. Hence, the $W$-matrix can be precomputed
before initiating Resinv. Moreover,
as mentioned in Remark~\ref{rem:SMW-variant},
this formulation of the SMW-formula, does
not require the storage of the matrices $F_1,\ldots,F_N$,
once  $W$ has been computed. In fact only one $F$-matrix \ejj{needs} to be
stored at a time, since the columns of $W$ can be computed
column-wise. \er{The precomputation can also be trivially parallelized since the columns of $W$ are independent.} This construction is summarized in
Algorithm~\ref{alg:WEP_resinv}. 

An important feature of the algorithm is that
$N$ can be treated as a parameter. Using a large $N$ implies
more computational work in the precomputation phase,
i.e., step 1-4 of Algorithm~\ref{alg:WEP_resinv}, 
since many Sylvester equations need
to \jk{be} solved. However, the quality of the preconditioner
is better 
\gm{and we expect the iterative method to convergence 
in fewer iterations for large $N$. More precisely,} less computation is required for the
iterative solves\footnote{This holds only if $N\ll n_xn_z$ as computation of \eqref{eq:alpha_system} otherwise is a significant part of Step~\ref{alg_step:WEP_resinv_lin_solve}.} in Step~\ref{alg_step:WEP_resinv_lin_solve}.
Hence, $N$ parameterizes a trade-off between computation time in
the initialization and in the iterative solves.
As is illustrated in Section~\ref{sec:numerics},
the best choice of $N$ in terms of total computation
time is a non-trivial problem. 

In order to further improve performance we use a result
regarding residual inverse iteration in \cite{Szyld:2012:LOCAL}.
More precisely, \cite[Theorem~9]{Szyld:2012:LOCAL}
states that the linear solves in Resinv
can be terminated in a way that preserves the
property that the convergence factor is proportional to the
shift-eigenvalue distance. It is proposed to
use the tolerance $\tau$ satisfying
\begin{equation} \label{eq:termination}  
  \|M(\gamma^{k+1})v_k-M(\sigma)\Delta v_k\|\le \tau \|M(\gamma^{k+1})v_k\|
\end{equation}
and $\tau=\mathcal{O}(|\gamma_*-\sigma|)$.
Although, we solve the linear system \eqref{eq:Ssigma_solve} inexactly,
the error propagates linearly to \eqref{eq:M_inverse}, 
and hence the tolerance \eqref{eq:termination} is natural also in our setting\er{.}


\subsection{Storage improvements}\label{subsec:storage}
The particular choice of SMW-formulation is
due to an observation in computational experiments,
that memory is a \ej{restricting} aspect in our approach. 
We have therefore selected the SMW-formulation
in order to reduce memory requirement at the cost
of an increased computation time.
Our approach requires the computation of
the solution to two Sylvester equations
per iteration (equation \eqref{eq:SMW_G} and
\eqref{eq:SMW-solution}).
The solution $X$ in \eqref{eq:SMW-solution}
could be computed by
forming a linear combination of $G, F_1, \ldots, F_N$, as in \eqref{eq:SMW-solution2}
but would require the storage of $F_1,\ldots, F_N$
which are full matrices in general.
In contrast to this, the matrices $E_1,\ldots,E_N$
have a structure that can be exploited.
More precisely, the matrix $C-\sum_{k=1}^N\alpha_k E_k$ required in
\eqref{eq:SMW-solution}
\gm{
can be computed efficiently 
by exploiting  
the structure of $V_k$ and 
$E_k$ defined in \eqref{eq:Ek_lm}.
}


\subsection{Circulant structure exploitation for Sylvester equation}\label{subsubsec:fft}
In order to use the suggested preconditioner we need to solve a number of Sylvester equations, with the Sylveser operator defined by \eqref{eq:Ldef}.
There are many methods available in the literature, both direct methods such as the Bartels-Steward algorithm \cite{Bartels:1972:LYAP} as well as iterative methods. \ej{See} \cite{Simoncini:2016:Computational}
for a recent survey of available methods.
These are general methods for solving the Sylvester equation.
However, our Sylvester equation has a particular structure which can be exploited further.
The approach is based on 
\gm{
the implicit diagonalization of
}
the coefficient matrices,
from which a closed form expression is available \cite{Simoncini:2016:Computational}. Consider the equation $AX + XB = C$ where $A$ and $B$ are diagonalizable. Then the solution $X$ is given by
\begin{align}\label{eq:diag_Sylv}
X = VYW^{-1}, \quad \text{ with } \quad [Y]_{p,q} = \er{\frac{[V^{-1}CW]_{p,q}}{[\Lambda_{A}]_p+[\Lambda_{B}]_q}},
\end{align}
where  $A=V\Lambda_{A} V^{-1}$, and $B=W\Lambda_{B} W^{-1}$.
\gm{In the general case, 
the application
of \eqref{eq:diag_Sylv} 
is expected to be expensive 
and numerically unstable.}
For the waveguide matrix equation \eqref{eq:WMATEQ} 
the matrix $A$ is circulant,
as it stems from the discretization with periodic boundary conditions. 
\gm{In particular}
it is diagonalized by the Fourier matrix \er{\cite[Theorem 4.8.2]{Golub:2013:MATRIX_edit}} whose action can be 
\gm{computed}
by FFT; the eigenvalues are also readily available in $\OOO(n_z\log(n_z))$ operations using FFT \cite{Golub:2013:MATRIX_edit}. The other matrix $B=D_{xx}$ is well studied and has both known eigenvalues and eigenvectors, 
the action of the latter can be 
\gm{computed}
in an efficient and stable way using the relation between Sine-/Cosine-transforms and FFT \cite[Lemma 6.1]{Demmel:1997:NLA_edit} \cite[Section~4.8]{Golub:2013:MATRIX_edit}.
The solution to the Sylvester operator in  \eqref{eq:Ldef}, i.e., 
\begin{align}\label{eq:syl3}
(D_{zz}+2\sigma D_z+(\sigma^2+\bar k) I_{n_z})X + XD_{xx}  = C,
\end{align}
can hence be \er{computed} by using \eqref{eq:diag_Sylv} since the 
action of $V$, $V^{-1}$, $W$, and $W^{-1}$ can be computed efficiently
and accurately using FFT, and the diagonals of $\Lambda_A$ and $\Lambda_B$ are 
available.
This exploitation of FFT leads to a
computation complexity
$\OOO(n_xn_z\log(n_xn_z))$
for the inversion of \eqref{eq:syl3},
cf. \cite[Section~6.7]{Demmel:1997:NLA_edit}.

\begin{algorithm} 
\caption{Resinv for WEP with preconditioned  Schur reformulation}\label{alg:WEP_resinv}
\SetKwInOut{Input}{input}\SetKwInOut{Output}{output}
\Input{Initial guess of the eigenpair $(\gamma_0,v_0)\in\CC\times\CC^{n_xn_z+2n_z}$, with $\norm{v_0} = 1$}
\Output{An approximation $(\gamma,v)\in\CC\times\CC^{n_xn_z+2n_z}$ of  $(\gamma_*,v_*)\in\CC\times\CC^{n_xn_z+2n_z}$}
\BlankLine
\nl \For{$k=1,2,\dots,N$}{
\nl Compute $F_k$ from \eqref{eq:SMW_F} with $E_k$ from \eqref{eq:Ek_lm}\\
\nl Compute the $k$-th column of $W$ as described in \eqref{eq:Wdef2}
}
\nl LU-factorize $W$\\
\nl \For{$k = 0,1,2,\dots$}{
\nl Compute new approximation of $\gamma_{k+1}$ from \eqref{eq:Rayleigh_quot}\\
\nl Compute the residual $r_k$ from \eqref{eq:residual_resinv}\\
\nl Compute $\tilde{r}_k$ from $r_k$ with \eqref{eq:Ssigma_solve_rhs}, using the sparsity of $C_1$ and the structure of $P(\sigma)^{-1}$ according to Remark \ref{rem:action_of_P}\\
\nl Compute $q$ from the linear system \eqref{eq:Ssigma_solve} with a preconditioned  iterative solver, where the preconditioner is applied to a vector $c$ as: \label{alg_step:WEP_resinv_lin_solve}
\begin{itemize}
\item Set $C$ such that $\vec(C)=c$ and compute $G$ from \eqref{eq:SMW_G}. Use $G$\\to compute $g$ from \eqref{eq:Wdef2}
\item Compute $\{\alpha_k\}_{k=1}^N$ by solving the linear system \eqref{eq:alpha_system} with the \\pre-factorized matrix $W$
\item Form the right hand side of \eqref{eq:SMW-solution} and solve the Sylvester equation. Vectorize the solution matrix $X$
\end{itemize}
\nl Compute the correction $\Delta v_k$ from $q$ using \eqref{eq:M_inverse}\\
\nl $v_{k+1} \gets v_k - \Delta v_k$\\
\nl $v_{k+1} \gets v_{k+1}/\|v_{k+1}\|$\\
}
\nl $\gamma \gets \gamma_k$, \ \ $v \gets v_k$\\
\end{algorithm}

\section{Numerical simulations}\label{sec:numerics}

In order to illustrate properties of our approach,
we now show the result of simulations  carried out 
in \matlab\ on a desktop computer.\footnote{Intel quad-core i5-5250U CPU 1.60 
GHz $\times$ $4$, with 16 GB RAM
using
\matlab\ 2015a (8.5.0.197613).}
Source code for the simulations are provided online to improve
reproducability.\footnote{URL: \url{https://www.math.kth.se/~eringh/software/wep/wep_code}}
\er{We use the waveguide illustrated in Figure \ref{fig:wg_challenge}a, and also described in \cite[Section~5.2]{Jarlebring:2015:WTIARTR}.}
This waveguide has many eigenvalues, oscillatory
eigenfunctions, and large discontinuity in the wavenumber,
which is constructed to be representative of a realistic situation. In the simulations we set the free parameter $\bar{k}$, as introduced in the waveguide matrix equation \eqref{eq:WMATEQ}, to $\bar{k} = \mean(K)$. Moreover, 
the size of the problem is
\ejj{denoted} $n$ and
defined as $n:=n_x n_z + 2n_z$, and the parametrization of the preconditioner, $N$, is defined by $N:=N_xN_z$. For implementation convenience we select $n_x = n_z +4$ and $N_x = N_z+4$.

We first illustrate the quality of the preconditioner (without incorporation into Resinv).
The relative error as function of GMRES-iteration is visualized in
Figure~\ref{fig:gmres_iteration_random_rhs}. We clearly see that the required
number of iterations decreases with $N$, which is expected since
the SMW-approximation error is smaller for larger $N$. 
Moreover, we see that a small $N$ normally generates a long transient phase.
The advantage of selecting a finer grid close to the boundary
as shown in Figure~\ref{fig:SMW-grid-for-S} is clear from the fact
that the convergence in
Figure~\ref{fig:gmres_iteration_random_rhs}a
is faster than the convergence in 
Figure~\ref{fig:gmres_iteration_random_rhs}b.
\begin{figure}[h]
  \begin{center}
    \subfigure[With the SMW-grid from Figure~\ref{fig:SMW-grid-for-S}, where $N_x=N_z+4$.]
    {\includegraphics{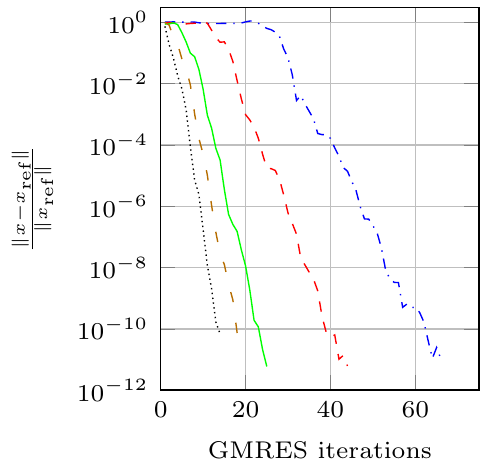}}\;\;\;\;\;%
    \subfigure[With a uniform coarser grid (cf. Figure~\ref{fig:SMW-grid-for-S}). Note that here $n_x=n_z$, and $N_x=N_z$.]
    {\includegraphics{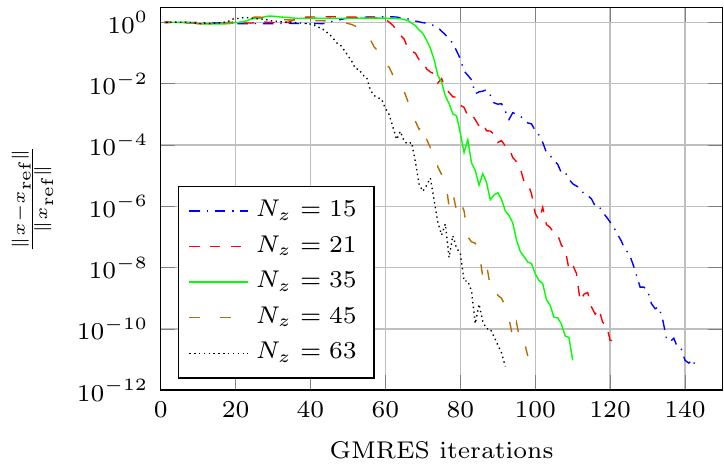}}
    \caption{
      GMRES convergence for solving $S(\sigma)x = c$ for different coarse grids in the SMW-approximation. The discretization is $n_z = 945$ and the error is measured as the relative error compared to a reference solution $x_\text{ref}$.
      \label{fig:gmres_iteration_random_rhs} 
    }
  \end{center}
\end{figure}

For the remaining simulations we measure the error of the approximation 
with an estimate of the relative residual norm
{\small \begin{align}\label{eq:rel_res_norm}
R(v,\gamma) = \frac{\|M(\gamma)v\|_2}{\sum_{k=0}^2 |\gamma|^k \|A_k\|_1 + \|C_1\|_1 + \|C_2^\trans\|_1+2|d_0| + \sum_{k=-p}^p(|s_{+,k}(\gamma)|+|s_{-,k}(\gamma)|)}
\end{align}}
analogous to estimates for other NEPs
\cite{Liao:2006:SOLVING,Higham:2008:BACKWARD}.
Algorithm~\ref{alg:WEP_resinv} is applied to this benchmark problem,
and the error is visualized in Figure~\ref{fig:resinv_error_and_solution}a.
We use $\sigma=-0.5-0.4i$, such that the algorithm converges
to the eigenvalue $\gamma\approx  -0.523 - 0.375i$. 
As expected from the GMRES-termination criteria \eqref{eq:termination},
we maintain linear convergence and a convergence factor
which is in the order of magnitude of the shift-eigenvalue distance.
The corresponding computed eigenfunction is visualized in Figure~\ref{fig:resinv_error_and_solution}b.

\begin{figure}[h]
  \begin{center}
    \subfigure[Error in relative residual norm \eqref{eq:rel_res_norm}.
    \ej{The prediction corresponds to convergence factor $|\gamma-\sigma|$.}]
{\includegraphics{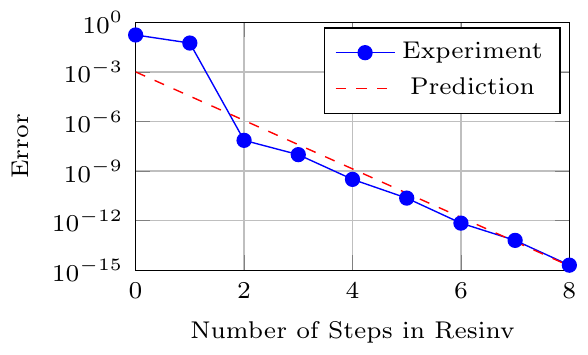}}
  \subfigure[Absolute value of the eigenfunction.]
    {\includegraphics{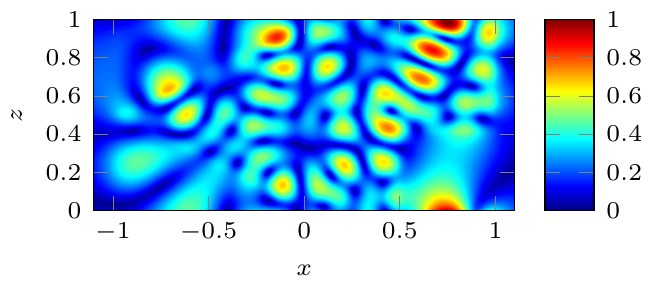}}

    \caption{
      Illustration of convergence of Resinv and the corresponding eigenfunction to \ej{eigenvalue} $\gamma\approx-0.523-0.375i$. In this simulation \ej{we use} the shift $\sigma=-0.5-0.4i$, the discretization $n_z = 2835$, and the preconditioner parameter $N_z = 21$.
      \label{fig:resinv_error_and_solution}
    }
  \end{center}
\end{figure}

Figure~\ref{fig:gmres_iter_count} shows
the number of required GMRES-iteration\jk{s for} 
one iteration of Resinv.
As expected from the approximation
properties of the preconditioner, \jk{a larger value $N$} 
implies fewer iterations. We observe an increase in the number of
required GMRES-iterations with increasing problem size. The
increase is however rather slow.

\jk{The choice of $N$ for parameterizing the trade-off between precomputation time and times in the solves,}
pointed \jk{out} in Section~\ref{subsec:smw_resinv}, is illustrated in
terms computation-time
in Figure~\ref{fig:resinv_times}a.
For this particular problem (and computing environment)
the best choice is  $N_z=35$. Note however, that several other choices such as $N_z=45$ and $N_z=63$ are almost as good. 
In profiling illustration in
Figure~\ref{fig:resinv_times}b, we see, as expected,
that increasing $N$, shifts computational effort into
to the precalculation phase, and that the
computational effort required for the
precalculation and \ejj{the other parts of the algorithm}
are of the same order of magnitude.

\begin{figure}[h]
  \begin{center}
  {\includegraphics{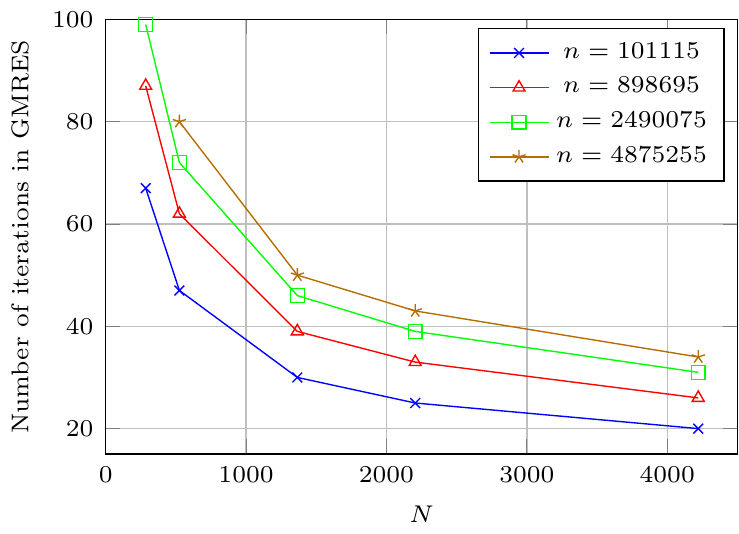}}
    \caption{Illustration of the increase in GMRES iterations. Evaluations for SMW-grid parameters $N_x=N_z+4$ and $N_z\in \{15, 21, 35, 45, 63\}$ where $N=N_xN_z$.
      \label{fig:gmres_iter_count}
    }
  \end{center}
\end{figure}

\begin{figure}[h]
  \begin{center}
    \subfigure[Time in seconds for solving the problem.]
  {\includegraphics{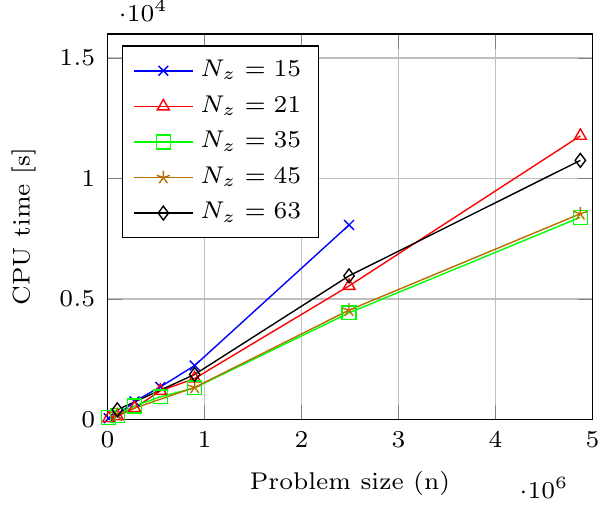}}\;\;\;\;\;%
    \subfigure[Percentage of time spent on precalculating the (SMW) $W$-matrix from \eqref{eq:Wdef2}.]
  {\includegraphics{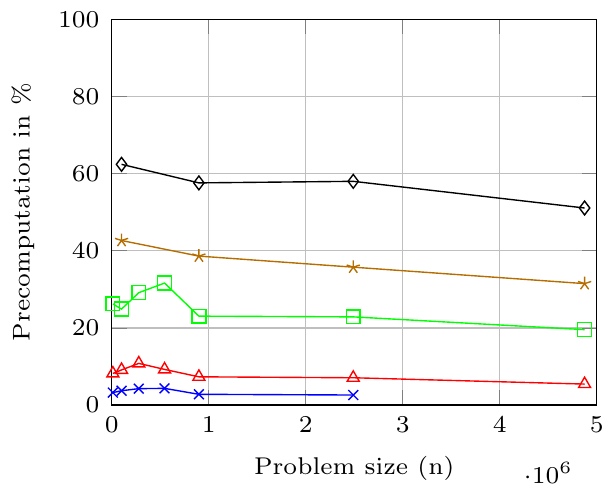}}
    %
    \caption{
      Time for the complete method described in Algorithm \ref{alg:WEP_resinv}, as a function of problem size. Problem size is the total size of the problem as presented in \eqref{eq:WEP}, that is $n_xn_z+2n_z$, where the GMRES-tolerance is select such that linear system is solved to full precision. This is plotted for
      \ejj{different $N$ values} 
      $N=N_x N_z$, and $N_x=N_z+4$.
      \label{fig:resinv_times}
    }
  \end{center}
\end{figure}

We now illustrate our experiences with the capabilities
of the full approach. 
We carry out simulations with the iterative methods GMRES and BiCGstab using the proposed preconditioner. 
In addition to this we compare with \ejj{GMRES combined with a preconditioner based on incomplete LU-factorization (ILU) \cite[Chapter~10]{Saad:1996:LINSYS_edit}}.
The computation times are given in Table~\ref{tbl:cpu-time}. 
For very large problems, GMRES is
not advantageous due to memory requirements. 
Dealing with this using restarted GMRES is not competitive due to long transient phases,
similar to those observed
in Figure~\ref{fig:gmres_iteration_random_rhs}.
GMRES requires more memory than BiCGStab, and the largest problem we manage to solve is computed with BiCGstab.\footnote{This size is larger than $28\cdot10^6\times28\cdot10^6$.}
However, GMRES was in general slightly faster in simulations with enough memory.
For large problems with GMRES, the size of the Krylov space needs to be carefully adjusted to stay within the available RAM. For most simulations the maximum size of the Krylov space is $100$ vectors, but for $n\approx 16 \cdot10^6$ only $30$ vectors is used  for $N_z=35$ and $N_z=21$. 
However, for the case $N_z=15$ no size of the Krylov subspace is found that is sufficiently large to reach convergence, but small enough to stay within the available RAM. Unfortunately,
in our experiments ILU required considerable memory resources
and we were not able to use it to solve very large problems.

A FEM-discretization of this problem was presented in \cite{Jarlebring:2015:WTIARTR}.
Our preconditioner can also be applied to solve the discretization with FEM,
by using the FD-preconditioner. \er{The observed convergence is similar to the previously observed convergence in the FD-case (cf. Figure~\ref{fig:gmres_iteration_random_rhs}a).
For instance, when our method is applied to a discretization
with $n_z=945$ and $N_z=21$, GMRES for the FD-problem needs $44$ iterations
and GMRES for the FEM-problem need $44$ iterations as well.}

\begin{table}[h]
  \begin{center}
  {
    \begin{tabular}{r||c|c|c||c|c|c||c|c}
           & \multicolumn{3}{c||}{GMRES} & \multicolumn{3}{c||}{BiCGStab} & \multicolumn{2}{c}{ILU (GMRES)} \\\hline 
  \diagbox[width=1.6cm, height=0.8cm]{~~~$n$}{\raisebox{-0.3ex}{$p$}}
  & $15$ & $21$ & $35$ & $15$ & $21$ & $35$ & $10^{-5}$ & $10^{-6}$\\
      \hline      \hline
      $101\,115$     & $63$ & $53$ & $79$ & $85$ & $66$ & $91$ & $78$ & $74$ \\
      $278\,775$     & $201$ & $169$ & $241$ & $324$ & $216$ & $246$ & $504$ & $474$ \\
      $898\,695$     & $542$ & $405$ & $509$ & $757$ & $520$ & $570$ & $\star$ & $\star$ \\
      $2\,490\,075$  & $1\,674$ & $1\,129$ & $1\,652$ & $2\,391$ & $1\,668$ & $1\,796$ & $\star$ & $\star$ \\
      $4\,875\,255$  & $3\,530$ & $2\,435$ & $3\,218$ & $4\,653$ & $3\,199$ & $3\,276$ & $\star$ & $\star$ \\
      $8\,054\,235$  & $6\,732$ & $4\,675$ & $6\,658$ & $10\,808$ & $7\,001$ & $7\,528$ & $\star$ & $\star$\\
      $16\,793\,595$ & $\star$ & $11\,171$ & $17\,116$ & $23\,047$ & $14\,865$ & $15\,983$ & $\star$ & $\star$
    \end{tabular}
    }
  \caption{Illustration of CPU-time in seconds for a set of different methods. The preconditioner parameter $p$ is $p = N_z$ for the proposed preconditioner, \ejj{
$N=N_xN_z$   
 where $N_x=N_z+4$;}
and for ILU $p=\epsilon$ the dropping tolerance. The symbol $\star$ denotes simulations which could not be executed due to insufficient RAM.
    \label{tbl:cpu-time}
    }
  \end{center}
\end{table}

\section{Concluding remarks and outlook}\label{sec:conclusion}
We have presented a new computational procedure specialized
for the WEP \eqref{eq:pde}, based on combining
the method for NEPs called Resinv and
an iterative method for linear systems.
The preconditioner for the iterative method
is based on an approximation leading
to a structure which can be exploited
with a matrix-equation version of SMW. 

There are many options for constructing the SMW approximation.
For instance, the space
used in the Galerkin approximation could
be selected in a number of ways.
Such a construction would necessarily need
to use 
\gm{sparsity or other matrix structures in order to solve \err{large-scale} problems.}

We have focused on one particular method
for NEPs: Resinv. One of the crucial features is
that a linear system corresponding to $M(\sigma)$
needs to be solved many times (for
a constant shift). The Resinv method
is not the only method that uses
the solution to many linear systems with
a fixed shift. This is also the case
for the
nonlinear Arnoldi method \cite{Voss:2004:ARNOLDI}
and the tensor infinite Arnoldi method
\cite{Jarlebring:2015:WTIARTR}. However,
inexact solves in Arnoldi-type methods are sometimes
problematic \cite{Simoncini:2003:INEXACTKRYLOV}, and further research would be
required in order to reliably and efficiently
use our preconditioner for these methods.
%

Although our approximation is justified with a
Galerkin approach, we have not provided any theoretical
convergence analysis.
The application of standard proof-techniques
for such an analysis, e.g., involving eigenvalues
and spectral condition numbers have 
not lead to a clear characterization of the error. 
Therefore, we believe that a convergence
analysis would require use of the regularity
of the eigenfunction, similiar to what
is used in multi-grid methods \cite{Dolean:2015:DOMAINDECOMP},
which is certainly beyond the scope of this paper.

Finally, we wish to point out that several
results in this paper may be of interest also
for other problems and other NEPs.
Most importantly, many NEPs arise naturally
from artificial boundary conditions.
Most artificial boundary conditions
has freedom regarding selection of boundary.
We can therefore select a rectangular
domain, i.e., similar
to the framework considered in this paper.

\section*{Acknowledgment}
The authors wish to thank Per Enqvist (KTH) and
Tobias Damm (TU Kaiserslautern) for comments and discussions
in the early developments of this result.
This research is supported by the
Swedish research council under Grant No. 621-2013-4640.
\bibliographystyle{plain}
\bibliography{eliasbib,misc}

\end{document}